\documentclass{amsart}
\usepackage[margin=1in]{geometry}

\usepackage{color,amssymb, comment, mathrsfs,tikz,upgreek, contour, soul, colonequals, mathdots,fancyvrb}
\usepackage[colorlinks, linkcolor={blue}]{hyperref}
\usepackage{tikz-cd}
\usepackage{cleveref}
\usetikzlibrary{cd}
\usetikzlibrary{fit,shapes.geometric}
\usetikzlibrary{matrix}
\usetikzlibrary{arrows}
\usepackage{enumitem}
\usepackage{mathtools}
\usepackage{graphicx}
\usepackage[all]{xy}
\usepackage[mathlines,pagewise]{lineno}
\usepackage[only,llbracket,rrbracket,llparenthesis,rrparenthesis]{stmaryrd} 
\usepackage{accsupp}
\allowdisplaybreaks

\theoremstyle{plain}
\newtheorem{theorem}{Theorem}[section]
\newtheorem{lemma}[theorem]{Lemma}
\newtheorem{proposition}[theorem]{Proposition}
\newtheorem{cor}[theorem]{Corollary}

\theoremstyle{definition}
\newtheorem{rem}[theorem]{Remark}
\newtheorem*{ack}{Acknowledgements}

\newtheorem{question}[theorem]{Question}

\crefname{rem}{remark}{remarks}
\Crefname{rem}{Remark}{Remarks}
\crefname{cor}{corollary}{corollaries}
\Crefname{cor}{Corollary}{Corollaries}

\numberwithin{equation}{section}

\newcommand{\D}{\textsf{\upshape D}}

\newcommand{\kk}{\Bbbk}

\newcommand{\xra}{\xrightarrow}
\newcommand{\id}{\operatorname{id}}

\newcommand{\Ext}{\operatorname{Ext}}

\newcommand{\Hom}{\operatorname{Hom}}

\newcommand{\RHom}{\operatorname{\mathbf{R}Hom}}

\newcommand{\Tor}{\operatorname{Tor}}

\newcommand{\cx}{\operatorname{cx}}
\newcommand{\injcx}{\operatorname{injcx}}
\newcommand{\q}{\mathbf{q}}
\newcommand{\m}{\mathfrak{m}}

\newcommand{\x}{\mathbf{x}}
\newcommand{\y}{\mathbf{y}}
\newcommand{\depth}{\operatorname{depth}}

\newcommand{\cidim}{\operatorname{CI-dim}}

\newcommand{\ciuid}{\operatorname{CI^*-id}}

\newcommand{\thick}{\operatorname{thick}}

\renewcommand{\max}{\operatorname{Max}}

\makeatletter
\DeclareFontEncoding{LS2}{}{\@noaccents}
\makeatother
\DeclareFontSubstitution{LS2}{stix}{m}{n}

\DeclareSymbolFont{largesymbolsstix}{LS2}{stixex}{m}{n}

\DeclareMathDelimiter{\lbrbrak}{\mathopen}{largesymbolsstix}{"EE}{largesymbolsstix}{"14}
\DeclareMathDelimiter{\rbrbrak}{\mathclose}{largesymbolsstix}{"EF}{largesymbolsstix}{"15}

\crefname{diagram}{diagram}{diagrams}
\crefname{diagram}{Diagram}{Diagrams}
\creflabelformat{diagram}{(#1) #2 #3}

\newcommand{\hsup}{\operatorname{hsup}}
\newcommand{\p}{\mathfrak{p}}
\newcommand{\A}{\mathcal{A}}
\newcommand{\Spec}{\operatorname{Spec}}
\newcommand{\pd}{\operatorname{pd}}
\newcommand{\hinf}{\operatorname{hinf}}
\newcommand{\qid}{\operatorname{qid}}
\newcommand{\fd}{\operatorname{fd}}

\newcommand{\qpd}{\operatorname{qpd}}
\renewcommand{\H}{\mathrm{H}}

\newcommand{\lotimes}{\otimes^{\mathbf{L}}}
\newcommand{\Supp}{\operatorname{Supp}}
\newcommand{\catb}{\sqsubset\mspace{-13mu}\sqsupset}

\newcommand{\Catsub}[3]{{\mathsf{#2}}_{#3}(#1)}
\newcommand{\Catsupsub}[4]{{\mathsf{#2}}^{\text{\upshape #3}}_{#4}(#1)}
\newcommand{\Catsup}[3]{{\mathsf{#2}}^{\text{\upshape #3}}(#1)}
\newcommand{\Db}[1][R]{\Catsub{#1}{D}{\catb}}
\newcommand{\Dfb}[1][R]{\Catsupsub{#1}{D}{f}{\catb}}

\newcommand{\Df}[1][R]{\Catsup{#1}{D}{f}}

\newcommand{\ciid}{\operatorname{CI-id}}

\newcommand{\cifd}{\operatorname{CI-fd}}

\newcommand{\amp}{\operatorname{amp}}

\newcommand{\CIHid}{\operatorname{CI}_{\mathrm{Hom}}\!\operatorname{-id}}

\newcommand{\perf}{\operatorname{Perf}}
\newcommand{\injperf}{\operatorname{InjPerf}}
\renewcommand{\q}{\mathfrak{q}}
\newcommand{\wh}{\widehat}



\author[Dey]{Souvik Dey}
\address[Souvik Dey]{Department of Mathematics \\ University of Arkansas \\ Fayetteville, AR 72701, U.S.A}
\email[]{souvikd@uark.edu}
\urladdr{https://orcid.org/0000-0001-8265-3301}

\author[Ferraro]{Luigi Ferraro}
\address[Luigi Ferraro]{School of Mathematical and Statistical Sciences \\ University of Texas Rio Grande Valley \\ Edinburg, TX 78539, U.S.A}
\email{luigi.ferraro@utrgv.edu}
\urladdr{https://faculty.utrgv.edu/luigi.ferraro}

\author[Gheibi]{Mohsen Gheibi}
\address[Mohsen Gheibi]{Department of Mathematics \\ Florida A\&M University \\ Tallahassee, FL 32307, U.S.A}
\email[]{mohsen.gheibi@famu.edu}

\begin{document}
\title{Complete intersection and quasi-homological dimensions}
\begin{abstract}
We prove that, over a commutative Noetherian ring, every finitely generated module of finite complete intersection dimension has finite quasi-projective dimension. Our proof adapts Bergh’s technique of reducing complexity, originally used to establish virtual smallness for complexes of finite complete intersection dimension over local rings, by successively enlarging the perfect locus of the cones and controlling their homology. This gives an affirmative answer to part of a question of Jorgensen, Takahashi, and the third author. We also establish a dual injective version of our result for rings admitting a dualizing complex. Finally, we show that, over such rings, the three injective analogues of complete intersection dimension appearing in the literature coincide for bounded complexes with finitely generated homology, while two of them coincide even without a dualizing complex. We also give an example showing that the finite-generation hypothesis is necessary, thereby providing a full answer to a question of Sather-Wagstaff. We conclude by studying the localization behavior of these dimensions, obtaining a partial answer to a question of Sather-Wagstaff and Totushek.
\end{abstract}

\maketitle

\section{Introduction}

Homological dimensions provide a fundamental connection between the properties of
modules and the singularities of the rings over which they are defined. For instance,
over a commutative Noetherian local ring, modules of finite projective dimension
exhibit many of the homological properties of modules over regular local rings.
The complete intersection dimension, introduced by Avramov, Gasharov, and Peeva,
extends projective dimension in a way that reflects the homological behavior of
modules over complete intersections. In particular, a local ring is a complete
intersection if and only if every finitely generated module has finite complete
intersection dimension.

More recently, Gheibi, Jorgensen, and Takahashi introduced the
\emph{quasi-projective dimension} in \cite{qpd}. Rather than enlarging the class of
modules allowed in a resolution, quasi-projective dimension enlarges the notion of
a resolution itself. A quasi-projective resolution of a module $M$ is a complex of
projective modules whose nonzero homology modules are finite direct sums of copies
of $M$. This construction arises naturally over complete intersections. Indeed, if
\[
R=Q/(\boldsymbol{x}),
\]
where $\boldsymbol{x}$ is a $Q$-regular sequence, and $M$ is an $R$-module having
finite projective dimension over $Q$, then tensoring a finite projective
$Q$-resolution of $M$ with $R$ yields a finite quasi-projective $R$-resolution of
$M$.

This connection led Gheibi, Jorgensen, and Takahashi to ask
\cite[Question~3.9]{qpd} whether every finitely generated module of finite complete
intersection dimension has finite quasi-projective dimension. The difficulty is
caused by the flat extension appearing in a quasi-deformation. More precisely,
finite complete intersection dimension provides a quasi-deformation
\[
R\longrightarrow R'\longleftarrow Q
\]
such that $R'\otimes_R M$ has finite projective dimension over $Q$. The preceding
construction therefore gives a finite quasi-projective resolution after passing to
$R'$, but it is not known in general whether finiteness of quasi-projective
dimension descends along a flat local homomorphism.

Our first main result gives an affirmative answer to the question without requiring
such a descent result and, in fact, holds over arbitrary commutative Noetherian rings. Namely, for every finitely generated $R$-module $M$, we show
that
\[
\cidim_R M<\infty
\quad\Longrightarrow\quad
\qpd_R M<\infty.
\]
When $R$ is local, the two dimensions are equal whenever they are finite. 

Our approach is inspired by Bergh's complexity-reducing construction
\cite{bergh}, which shows that a homologically finite complex of finite complete
intersection dimension over a local ring is virtually small. We refine and
globalize this construction, retaining control of homology while successively
enlarging the locus on which the resulting complexes are perfect. This allows
us to produce a perfect complex whose homology has the form required for a
quasi-projective resolution.

We also study the injective counterpart. Gheibi introduced the
\emph{quasi-injective dimension} in \cite{qid} as the natural dual of
quasi-projective dimension. Using Grothendieck duality, we obtain a dual version
of our first result: if $R$ is a commutative Noetherian ring admitting a
dualizing complex and $M$ is a finitely generated $R$-module, then
\[
\ciuid_R M<\infty
\quad\Longrightarrow\quad
\qid_R M<\infty.
\]
When $R$ is local, these dimensions agree whenever they are finite.

The injective setting also leads us to compare the three injective analogues
of complete intersection dimension appearing in the literature: complete intersection injective dimension
$\ciid$, upper complete intersection injective dimension $\ciuid$
\cite{ciInj}, and complete intersection Hom-injective dimension $\CIHid$
\cite{cihid}. We devote the final section of the paper to
the relationship among these invariants. If $R$ admits a dualizing complex and $M$ is a bounded complex with
finitely generated homology, we prove
\[
\ciid_R M=\ciuid_R M=\CIHid_R M.
\]
Without assuming the existence of a dualizing complex, we still obtain
\[
\ciid_R M=\ciuid_R M
\]
for every bounded complex with finitely generated homology over a local ring; moreover, if
$\CIHid_RM$ is finite, then all three dimensions coincide.

Together, these results provide a full answer to a question of
Sather-Wagstaff \cite[Question~2.9]{ciInj}: the equality
$\ciid_R M=\ciuid_R M$ holds for bounded complexes with finitely generated
homology, while an example arising from a construction of Ferrand and Raynaud
shows that it can fail without the finite-generation hypothesis.
Finally, we answer a localization question of Sather-Wagstaff and Totushek
\cite[Question~3.9]{cihid} for bounded complexes with finitely generated
homology over rings with a dualizing complex.

The paper is organized as follows. Section~2 contains background and notation.
In Section~3 we relate complete intersection dimension and quasi-projective
dimension. Section~4 develops the corresponding theory for upper complete
intersection injective dimension and quasi-injective dimension. In Section~5
we compare the three complete intersection injective dimensions, study their
behavior under completion and localization, and address the questions mentioned
above.

\begin{ack}
Luigi Ferraro was partly supported by the Simons Foundation grant MPS-TSM-00007849. Discussions with ChatGPT 5.6 Sol helped complete parts of this work. 
\end{ack}

\section{Background}

Throughout the paper, $R$ denotes a commutative Noetherian ring. When $R$ is assumed to be local, we will denote by $\m$ its unique maximal ideal and by $\kk$ its residue field. If $\p$ is a prime ideal of $R$, we denote by $\kk(\p)$ the residue field of $R_\p$.

We use homological indexing for complexes. Given an $R$-complex $M$, set
\[
\sup M=\sup\{i\in\mathbb Z\mid M_i\neq 0\}
\quad\text{and}\quad
\inf M=\inf\{i\in\mathbb Z\mid M_i\neq 0\}.
\]
To distinguish these from the corresponding homological bounds, we write
\[
\hsup M=\sup\{i\in\mathbb Z\mid H_i(M)\neq 0\}
\quad\text{and}\quad
\hinf M=\inf\{i\in\mathbb Z\mid H_i(M)\neq 0\}.
\]

We also set $\amp M=\hsup M-\hinf M$.

The derived category of $R$-modules is denoted by $\D(R)$. We use $\Db$ for the full subcategory consisting of
complexes with bounded homology. We denote by $\Df$ the full subcategory
of complexes whose homology modules are finitely generated in every degree,
and set
\[
\Dfb=\Df\cap\Db.
\]

We next recall the homological dimensions and categorical constructions
that will be used throughout the paper.

\subsection*{Complexity and injective complexity}

Let $M\in\Dfb$. The $i$th Betti number and the $i$th Bass number of $M$
are, respectively,
\[
\beta_i^R(M)=\dim_\kk \Tor_i^R(\kk,M)
\quad\text{and}\quad
\mu_R^i(M)=\dim_\kk \Ext_R^i(\kk,M).
\]
The \emph{complexity} of $M$ measures the polynomial rate of growth of its
Betti numbers and is defined by
\[
\cx_R M
=
\inf\left\{
d\in\mathbb N
\ \middle|\
\beta_i^R(M)\leq a i^{d-1}
\text{ for some }a\in\mathbb R
\text{ and all }i\gg0
\right\}.
\]
Similarly, the \emph{injective complexity} of $M$ is
\[
\injcx_R M
=
\inf\left\{
d\in\mathbb N
\ \middle|\
\mu_R^i(M)\leq a i^{d-1}
\text{ for some }a\in\mathbb R
\text{ and all }i\gg0
\right\}.
\]
Thus $\cx_RM=0$ precisely when the Betti numbers of $M$ eventually vanish,
and $\injcx_RM=0$ precisely when its Bass numbers eventually vanish.

Suppose now that $R$ admits a dualizing complex $D$. For $M\in\Dfb$ we set the following notation
\[
M^\dagger:=\RHom_R(M,D).
\]
Grothendieck duality gives a contravariant equivalence on $\Dfb$, and the
Betti numbers of $M^\dagger$ and the Bass numbers of $M$ have the same
asymptotic growth. Consequently,
\begin{equation}\label{eq:injcx-duality}
\injcx_R M=\cx_R M^\dagger.
\end{equation}
See, for example, \cite[Remark~7.1.2]{AIM}.

\subsection*{Thick subcategories}

We recall some terminology concerning generation in triangulated categories.
A full subcategory $\mathcal T$ of $\D(R)$ is called \emph{thick} if it is closed under shifts, exact triangles, and direct summands.

For $M\in\D(R)$, we denote by
\[
\thick_{\D(R)}M
\]
the smallest thick subcategory of $\D(R)$ containing $M$.  We refer to \cite[Section~2]{bergh} for this terminology and
for an explicit inductive construction of $\thick_{\D(R)}M$.

\subsection*{Quasi-projective dimension}

Let $M$ be a nonzero $R$-module. A \emph{quasi-projective resolution} of
$M$ is a complex $F$ of $R$-modules satisfying the following conditions:
\begin{enumerate}
    \item $F_i=0$ for $i\ll0$;
    \item each $F_i$ is a projective $R$-module;
    \item for every $i\in\mathbb Z$, there is an integer $a_i\geq0$ such that
    \[
    H_i(F)\cong M^{\oplus a_i},
    \]
    and $a_i\neq0$ for at least one $i$.
\end{enumerate}
The \emph{quasi-projective dimension} of $M$ is
\[
\qpd_RM
=
\inf\left\{
\sup F-\hsup F
\ \middle|\
F\text{ is a quasi-projective resolution of }M
\text{ with }\sup F<\infty
\right\}.
\]
By convention, $\qpd_R0=-\infty$.

\subsection*{Quasi-injective dimension}\label{qpdsubsection}

Dually, a \emph{quasi-injective resolution} of a nonzero $R$-module $M$
is a complex $I$ of $R$-modules such that
\begin{enumerate}
    \item $I_i=0$ for $i\gg0$;
    \item each $I_i$ is an injective $R$-module;
    \item for every $i\in\mathbb Z$, there is an integer $a_i\geq0$ for which
    \[
    H_i(I)\cong M^{\oplus a_i},
    \]
    with $a_i\neq0$ for at least one $i$.
\end{enumerate}
The \emph{quasi-injective dimension} of $M$ is defined by
\[
\qid_RM
=
\inf\left\{
\hinf I-\inf I
\ \middle|\
I\text{ is a quasi-injective resolution of }M
\text{ with }\inf I>-\infty
\right\},
\]
and we set $\qid_R0=-\infty$.

\subsection*{Complete intersection dimensions}

We finish by collecting the complete intersection dimensions needed below.
A \emph{quasi-deformation} of $R$ is a diagram of local homomorphisms
\[
R\xrightarrow{\alpha}R'\xleftarrow{p}Q
\]
in which $\alpha$ is flat and $p$ is surjective with kernel generated by a
$Q$-regular sequence. Such a quasi-deformation will be called
\emph{exceptional} if $R'$ has Gorenstein formal fibers and the closed fiber
$R'/\m R'$ is Gorenstein. In particular, the first condition is satisfied
when $R'$ admits a dualizing complex.

For a nonzero $R$-complex $M$, we consider the following invariants:
\begin{align*}
&\cidim_RM
=
\inf\left\{
\pd_Q(R'\otimes_R M)-\pd_QR'
\ \middle|\
R\longrightarrow R'\longleftarrow Q
\text{ is a quasi-deformation}
\right\},\\
&\cifd_RM
=
\inf\left\{
\fd_Q(R'\otimes_R M)-\pd_QR'
\ \middle|\
R\longrightarrow R'\longleftarrow Q
\text{ is a quasi-deformation}
\right\},\\
&\ciid_RM
=
\inf\left\{
\id_Q(R'\otimes_R M)-\pd_QR'
\ \middle|\
R\longrightarrow R'\longleftarrow Q
\text{ is a quasi-deformation}
\right\},\\
&\ciuid_RM
=
\inf\left\{
\id_Q(R'\otimes_R M)-\pd_QR'
\ \middle|\
R\longrightarrow R'\longleftarrow Q
\text{ is an exceptional quasi-deformation}
\right\},\\
&\CIHid_RM=\inf\left\{
\id_Q\RHom_R(R',M)-\pd_QR'
\ \middle|\
R\longrightarrow R'\longleftarrow Q
\text{ is a quasi-deformation}
\right\}.
\end{align*}

The invariant $\ciuid_RM$ is the \emph{upper complete intersection
injective dimension}, introduced in \cite[Definition~2.6]{ciInj}, while
$\CIHid_RM$ is the \emph{complete intersection Hom-injective dimension}
of \cite[Definition~3.1]{cihid}. We will also use the fact that, for
$M\in\Dfb$, one has
\[
\cidim_RM=\cifd_RM.
\]

\section{Complete intersection dimension and quasi-projective dimension}

In this section we will show that a finitely generated module of finite complete intersection dimension has finite quasi-projective dimension. We first prove a preliminary lemma which will be fundamental in proving the result mentioned above. 

\begin{lemma}\label{lem1}
    Let $\A$ be an abelian category. Let $X, Y \in \D(\A)$ be such that there is an exact triangle 
    \[
    X \to \Sigma^nX \to Y \to 
    \]
    where $n>\amp X+1$. Then,
    \[
    \H(Y)\cong\Sigma^n\H(X)\oplus\Sigma\H(X).
    \]
\end{lemma}

\begin{proof} For every $i\in \mathbb Z$, we have the following exact sequence
    \[
    \H_i(X)\to \H_{i-n}(X)\to \H_i(Y)\to \H_{i-1}(X)\to \H_{i-n-1}(X).
    \]
    For every $i\in \mathbb Z$ either $\H_{i-1}(X)=0$ or $\H_{i-n}(X)=0$, indeed the difference of the indices is $(i-1)-(i-n)=n-1>\amp X$.
    
    Therefore the maps $\H_i(X)\to \H_{i-n}(X)$ and $\H_{i-1}(X)\to \H_{i-n-1}(X)$ are zero. Consequently, the previous exact sequence reduces to
    \[
    0\to \H_{i-n}(X)\to \H_i(Y)\to \H_{i-1}(X)\to 0.
    \]
Since either $\H_{i-1}(X)=0$ or $\H_{i-n}(X)=0$, it follows that $\H_i(Y)$ is isomorphic to either $\H_{i-1}(X)$ or $\H_{i-n}(X)$ and we can write $\H_i(Y)\cong\H_{i-n}(X)\oplus\H_{i-1}(X)$ for every $i\in\mathbb{Z}$, providing the asserted isomorphism.
\end{proof}

In the following two remarks we will collect results that follow from the proof of \cite[Theorem 3.2]{bergh} and that will be used to prove the main theorem of this section.

\begin{rem}\label{rem:bergh}
If $R$ is local and $\cidim_RM<\infty$, then there exists $\eta\in\Ext_R(M,M)$ of positive degree $d$ such that the map induced by multiplication by $\eta$
\[
\Ext_R(M,\kk)\rightarrow\Ext_R(M,\kk)
\]
is eventually injective. The element $\eta$ corresponds to a map $M\xra{f_\eta}\Sigma^dM$ in $\D(R)$.
\end{rem}

The next remark follows by the same argument as in the proof of \cite[Theorem 3.2]{bergh}. 

\begin{rem}\label{rem:cx}
If $R$ is local and there is an exact triangle in $\Dfb$
\[
M\rightarrow\Sigma^nM\rightarrow C\rightarrow
\]
with a positive $n$, such that
\begin{itemize}
\item $\cx_R M<\infty$.
\item The Poincar\'{e} series of $M$ is rational.
\item The map of graded $R$-modules $\Ext_R(M,\kk)\rightarrow\Ext_R(M,\kk)$ induced by $M\rightarrow\Sigma^nM$ is eventually injective.
\end{itemize}
Then, $\cx_RC=\cx_RM-1$. We point out that the first two conditions are satisfied if $\cidim_RM<\infty$.
\end{rem}

We are now ready to prove the main result of this section which strengthens \cite[Theorem 3.2]{bergh}.

We set the following notation for $M\in\Dfb$
\[
\perf M=\{\p\in\Spec R\mid M_\p\in\thick_{\D(R_\p)}R_\p\}.
\]

\begin{theorem}\label{thm:global}
Let $R$ be a commutative Noetherian ring not necessarily local. Let $M\in\Dfb$ with $\cidim_RM<\infty$ and $M\not\in\thick_{\D(R)}R$. Let $\p\not\in\perf M$, then for every integer $n\in[0,\cx_{R_\p}M_\p]$ there exists a nonzero $X_n\in\Dfb$ such that
\begin{enumerate}
\item $X_n\in\thick_{\D(R)}M$.
\item $\cidim_R X_n<\infty$.
\item $\cx_{R_\p}(X_n)_\p=n$.
\item There are positive integers $i_1,\ldots, i_j$ such that
\[
\H(X_n)\cong\Sigma^{i_1}\H(M)\oplus\cdots\oplus\Sigma^{i_j}\H(M).
\]
\item $\Supp_R X_n=\Supp_RM$.
\item If $n\geq1$ there is an exact triangle
\[
X_n\rightarrow\Sigma^{l_n}X_n\rightarrow X_{n-1}\rightarrow
\]
with $l_n>\amp X_n+1$.
\item If $n\geq1$, then $\perf X_n\subseteq \perf X_{n-1}$. 
\end{enumerate}
In particular $\perf M\subsetneq\perf X_0$.
\end{theorem}
\begin{proof}
Note that $\cx_{R_\p}M_\p>0$. We set $X_{\cx_{R_\p}M_\p}=M$. Assume that $X_n$ has been constructed and that $n>0$, we show how to construct $X_{n-1}$.

We choose $\eta\in\Ext_{R_\p}((X_n)_\p,(X_n)_\p)$ of positive degree $d_n$ as in \Cref{rem:bergh}. Let $N$ be such that $l_n\colonequals Nd_n>\amp X_n+1$. Regard $\eta$ as an element of $\Ext_R(X_n,X_n)_\p$, then there is a $\nu\in\Ext_R(X_n,X_n)$ and $s\not\in\p$ such that $\nu_\p=s\eta$. Complete $f_{\nu^N}$ to an exact triangle
\[
X_n\rightarrow\Sigma^{l_n}X_n\rightarrow X_{n-1}\rightarrow
\]
We show that $X_{n-1}$ has the desired properties. Note that $X_{n-1}$ is nonzero, otherwise $X_n$ would be isomorphic to $\Sigma^{l_n}X_n$, which is not possible.

To prove (1) it suffices to notice that
\[
X_{n-1}\in\thick_{\D(R)}X_n\subseteq\thick_{\D(R)}M.
\]

Part (1) and \cite[Lemma 3.1]{bergh} applied to the localization of the triangle above shows that $\cidim_RX_{n-1}<\infty$, proving (2).

Since $s$ is a unit in $R_\p$, it follows that the map $\Ext_{R_\p}((X_{n})_\p,\kk(\p))\rightarrow\Ext_{R_\p}((X_{n})_\p,\kk(\p))$ induced by multiplication by $\nu_\p^N$ is eventually injective, therefore (3) follows from \Cref{rem:cx} and the assumptions on $X_n$.

By \Cref{lem1} 
\[
\H(X_{n-1})\cong\Sigma^{l_n}\H(X_n)\oplus\Sigma\H(X_n),
\]
and since $X_n$ satisfies (4), so does $X_{n-1}$.

Condition (5) holds by \cite[Lemma 2.4.2]{josh}, but it also follows directly from (4).

Part (6) follows directly from our construction.

If $\q\in\perf X_n$, then the exact triangle in (6) shows that $\q\in\perf X_{n-1}$, yielding (7).

In particular $\perf M\subseteq\perf X_0$, but since $\cx_{R_\p}(X_0)_\p=0$, it follows that $\p\in\perf X_0$, giving the strict containment.
\end{proof}

\begin{cor}
Let $R$ be a commutative Noetherian ring not necessarily local. Let $M$ be a finitely generated $R$-module. If $\cidim_RM<\infty$, then $\qpd_RM<\infty$.
\end{cor}
\begin{proof}
If $M\in\thick_{\D(R)}R$, then there is nothing to prove, so assume the contrary. Set $Y_0=M$, then by \Cref{thm:global} there is a $Y_1\in\Dfb$ such that 
\begin{enumerate}
\item $\cidim_RY_1<\infty$,
\item $\H_i(Y_1)\cong M^{a_{i,1}}$ for some $a_{i,1}$ not all zero. 
\item $\perf Y_0\subsetneq\perf Y_1$.
\end{enumerate}
If $Y_1\not\in\thick_{\D(R)}R$, then applying \Cref{thm:global} to $Y_1$ and iterating possibly infinitely many times, yields a sequence $Y_0,Y_1,Y_2,\ldots$ of objects in $\Dfb$ such that
\begin{enumerate}
\item $\cidim_RY_t<\infty$,
\item $\H(Y_t)\cong M^{a_{i,t}}$ for some $a_{i,t}$ not all zero.
\item $\perf Y_{t-1}\subsetneq\perf Y_t$.
\end{enumerate}
Since $\perf Y_t$ is open by \cite[Lemma 2.3]{SriRyo} for all $t$ and $\Spec R$ is Noetherian, it follows that this ascending chain stabilizes, i.e. there is a $t$ such that $\perf Y_t=\Spec R$. Therefore $Y_t$ is locally perfect, and by \cite[Theorem 17.3.28]{LarsBook} it follows that so is $Y_t$. A bounded semi-projective resolution of $Y_t$ is a bounded quasi-projective resolution of $M$, yielding $\qpd_RM<\infty$.
\end{proof}

The previous corollary answers the last part of \cite[Question 3.9]{qpd}.

\begin{cor}
If $R$ is local and $M$ a finitely generated $R$-module with $\cidim_RM<\infty$, then 
\[
\cidim_RM=\qpd_RM.
\]
\end{cor}
\begin{proof}
This follows immediately from the previous corollary since both dimensions satisfy an Auslander-Buchsbaum formula.
\end{proof}

\begin{rem}
The flat descent problem for quasi-projective dimension from \cite[Question 3.9]{qpd} remains an open problem, even for $R\rightarrow\wh R$.
\end{rem}

\section{Upper complete intersection injective dimension and quasi-injective dimension}
In this section we explore dual statements to the ones obtained in the previous section. We define the upper complete intersection dimension over rings that are not necessarily local by
\[
\ciuid_RM\colonequals\sup\{\ciuid_{R_\m}M_\m\mid\m\in\max R\}.
\]
For a ring with a dualizing complex $D$ we set the following notation
\[
\injperf M=\{\p\in\Spec R\mid M_\p\in\thick_{\D(R_\p)}D_\p\},
\]
and since $\injperf M=\perf\RHom_R(M,D)$, it follows that these sets are open by \cite[Lemma 2.3]{SriRyo}.

We start with a dual result to \Cref{thm:global}.

\begin{theorem}\label{thm:globalinj}
Let $R$ be a commutative Noetherian ring (not necessarily local) with a dualizing complex $D$. Let $M\in\Dfb$ with $\ciuid_RM<\infty$ and $M\not\in\thick_{\D(R)}D$. Let $\p\not\in\injperf M$, then for every integer $n\in[0,\injcx_{R_\p}M_\p]$ there exists a nonzero $X_n\in\Dfb$ such that
\begin{enumerate}
\item $X_n\in\thick_{\D(R)}M$.
\item $\ciuid_R X_n<\infty$.
\item $\injcx_{R_\p}(X_n)_\p=n$.
\item There are positive integers $i_1,\ldots, i_j$ such that
\[
\H(X_n)\cong\Sigma^{i_1}\H(M)\oplus\cdots\oplus\Sigma^{i_j}\H(M).
\]
\item $\Supp_R X_n=\Supp_RM$.
\item If $n\geq1$ there is an exact triangle
\[
X_n\rightarrow\Sigma^{l_n}X_n\rightarrow X_{n-1}\rightarrow
\]
with $l_n>\amp X_n+1$.
\item If $n\geq1$, then $\injperf X_n\subseteq \injperf X_{n-1}$. 
\end{enumerate}
In particular $\injperf M\subsetneq\injperf X_0$.
\end{theorem}

\begin{proof}
Note that $\injcx_{R_\p}M_\p>0$. We set $X_{\injcx_{R_p}M_\p}=M$. Assume that $X_n$ has been constructed and that $n>0$, we show how to construct $X_{n-1}$. Whenever we apply $(-)^\dagger$ to an object of $\D(R_\p)$ we mean $\RHom_{R_\p}(-,D_\p)$. By \cite[Corollary 4.6(b)]{ciInj} $\cidim_R(X_n)_\p^\dagger<\infty$. We choose 
\[
\eta\in\Ext_{R_\p}((X_n)_\p^\dagger,(X_n)_\p^\dagger)
\]
of positive degree $d_n$ as in \Cref{rem:bergh}. Let $N$ be such that $l_n\colonequals Nd_n>\amp X_n+1$. Regard $\eta$ as an element of $\Ext_R(X_n^\dagger, X_n^\dagger)_\p$, then there is a $\nu\in\Ext_R(X_n^\dagger,X_n^\dagger)$ and $s\not\in\p$ such that $\nu_\p=s\eta$. Complete $f_{\nu^N}$ to an exact triangle
\[
X_n^\dagger\rightarrow\Sigma^{l_n}X_n^\dagger\rightarrow Y_{n-1}\rightarrow.
\]
We set $X_{n-1}\colonequals \Sigma^{l_n+1}Y_{n-1}^\dagger$ and apply the triangulated functor $(-)^\dagger$ to the previous exact triangle. After invoking Grothendieck duality and shifting, this yields the following exact triangle
\begin{equation}\label{eq:triangle3}
X_n\rightarrow\Sigma^{l_n}X_n\rightarrow\ X_{n-1}\rightarrow.
\end{equation}
We show that $X_{n-1}$ satisfies the required properties.

We note that $X_{n-1}\in\thick_{\D(R)}X_n\subseteq\thick_{\D(R)}M$, proving (1).

By \cite[Lemma 3.1]{bergh} $\cidim_R Y_{n-1}<\infty$. It now follows from \cite[Corollary 4.6(a)]{ciInj} and Grothendieck duality that $\ciuid_RX_{n-1}<\infty$, thus (2) follows.

Since $s$ is a unit in $R_\p$, it follows that the map $\Ext_{R_\p}((X_{n})^\dagger_\p,\kk(\p))\rightarrow\Ext_{R_\p}((X_{n})^\dagger_\p,\kk(\p))$ induced by multiplication by $\nu_\p^N$ is eventually injective.It follows from \Cref{rem:cx} that $\cx_{R_\p}(Y_{n-1})_\p=\cx_R(X_n)_\p^\dagger-1$. We note that
\[
\injcx_{R_\p}(X_{n-1})_\p=\cx_{R_\p}(Y_{n-1})_\p,\quad \injcx_{R_\p} (X_n)_\p=\cx_{R_\p}(X_n)_\p^\dagger,
\]
yielding (3).

Applying \Cref{lem1} to \eqref{eq:triangle3} one gets
\[
\H(X_{n-1})\cong\Sigma^l\H(X_n)\oplus\Sigma\H(X_n),
\]
and therefore (4) follows by induction.

Condition (5) holds by \cite[Lemma 2.4.2]{josh}, but it also follows directly from (4).

Part (6) is immediate from our construction.

If $\q\in\injperf X_n$, then $(X_n)_\q\in\thick_{\D(R_\q)}D_\q$. By \eqref{eq:triangle3} $(X_{n-1})_\q\in\thick_{\D(R_\q)}(X_n)_\q\subseteq\thick_{\D(R_\q)}D_\q$ yielding (7).

In particular $\injperf M\subsetneq\injperf X_0$ since $\injcx_{R_\p}(X_0)_\p=0$, and therefore $\p\in\injperf X_0$,
\end{proof}

\begin{cor}\label{cor:ciuid}
Let $R$ be a commutative Noetherian ring not necessarily local. Let $M$ be a finitely generated $R$-module. If $R$ has a dualizing complex and $\ciuid_RM<\infty$, then $\qid_RM<\infty$.
\end{cor}
\begin{proof}
Let $D$ be a dualizing complex for $R$. If $M\in\thick_{\D(R)}D$, then there is nothing to prove, so assume the contrary. Set $Y_0=M$, then by \Cref{thm:globalinj} there is a $Y_1\in\Dfb$ such that 
\begin{enumerate}
\item $\ciuid_RY_1<\infty$,
\item $\H_i(Y_1)\cong M^{a_{i,1}}$ for some $a_{i,1}$ not all zero. 
\item $\injperf Y_0\subsetneq\injperf Y_1$.
\end{enumerate}
If $Y_1\not\in\thick_{\D(R)}D$, then applying \Cref{thm:globalinj} to $Y_1$ and iterating possibly infinitely many times, yields a sequence $Y_0,Y_1,Y_2,\ldots$ of objects in $\Dfb$ such that
\begin{enumerate}
\item $\ciuid_RY_t<\infty$,
\item $\H(Y_t)\cong M^{a_{i,t}}$ for some $a_{i,t}$ not all zero.
\item $\injperf Y_{t-1}\subsetneq\injperf Y_t$.
\end{enumerate}
Since $\injperf Y_t$ is open for all $t$ and $\Spec R$ is Noetherian, it follows that this ascending chain stabilizes, i.e. there is a $t$ such that $\injperf Y_t=\Spec R$. Therefore $\id_{R_\q}(Y_t)_\q<\infty$ for all $\q\in\Spec R$, and therefore $\id_RY_t<\infty$ by \cite[Theorem 17.3.28]{LarsBook} and Grothendieck duality. A bounded semi-injective resolution of $Y_t$ is a bounded quasi-injective resolution of $M$, yielding $\qid_RM<\infty$.
\end{proof}

\begin{cor}\label{cor:ci*id&qid}
If $R$ is a local ring with a dualizing complex and $M$ a finitely generated $R$-module with $\ciuid_RM<\infty$, then 
\[
\ciuid_RM=\qid_RM.
\]
\end{cor}
\begin{proof}
This follows immediately from the previous corollary since both dimensions satisfy a Bass formula, see \cite[Proposition 2.11(b)]{ciInj} and \cite[Theorem 3.2]{qid}.
\end{proof}

\begin{question}
Does \Cref{cor:ci*id&qid} hold for rings that do not necessarily have a dualizing complex? A possible approach would be to reduce to the completion by provig that $\qid_{\wh R}\wh M<\infty$ implies $\qid_RM<\infty$. More generally we ask whether flat descent holds for the quasi-injective dimension.
\end{question}

\section{Comparisons of complete intersection injective dimensions}

It is natural to ask whether \Cref{thm:globalinj} and therefore \Cref{cor:ciuid} hold if $\CIHid_RM<\infty$ or $\ciid_RM<\infty$, where these dimensions are defined over nonlocal rings as
\[
\CIHid_RM\colonequals\sup\{\CIHid_{R_\m}M_\m\mid\m\in\max R\},
\]
\[
\ciid_RM\colonequals\sup\{\ciid_{R_\m}M_\m\mid\m\in\max R\}.
\]
We will show that this is indeed the case since under the hypotheses of \Cref{thm:globalinj} one has
\[
\ciuid_RM=\CIHid_RM=\ciid_RM.
\]
It suffices to establish the equality above over local rings.

We first start by exploring the behavior of the complete intersection Hom-injective dimension under Grothendieck duality. 

\begin{theorem}\label{thm:duality}
Let $R$ be a local ring. Let $D$ be a normalized dualizing complex of $R$, and let $M\in\Dfb$. Then,
\begin{enumerate}
\item $\CIHid_RM=\cidim_RM^\dagger$.
\item $\cidim_RM=\CIHid_RM^\dagger$.
\end{enumerate}
\end{theorem}
\begin{proof}
Since (2) follows from (1) by Grothendieck duality, we only need to prove (1). Let $E$ be the injective hull of $\kk$. We introduce the following notation:
\[
(-)^\vee\colonequals\RHom_R(-,E).
\]
Let $\x=x_1,\ldots,x_n$ be a generating set for $\m$, and let $K$ be the Koszul complex on $\x$. Set
\[
Z\colonequals\RHom_R(K,M).
\]
By \cite[Proposition 11.4.6, Theorem 13.3.29 and Proposition 14.3.2]{LarsBook} the complex $Z$ is derived $\m$-torsion. By \cite[Proposition 18.2.38]{LarsBook} $Z^\dagger\simeq Z^\vee$. By \cite[Theorem 12.3.25(a)]{LarsBook} it follows that
\begin{equation}\label{eq:iso}
K\lotimes_R M^\dagger\simeq\RHom_R(\RHom_R(K,M),D)= Z^\dagger\simeq Z^\vee.
\end{equation}
Next we prove that
\begin{equation}\label{eq:ci}
\cidim_R K\lotimes_R M^\dagger=\cidim_R M^\dagger +n.
\end{equation}
Let $R\rightarrow R^\prime\leftarrow Q$ be a quasi deformation. If we show that 
\begin{equation}\label{eq:ci2}
\pd_Q R^\prime\lotimes_R(K\lotimes_R M^\dagger)=\pd_Q R^\prime\lotimes_R M^\dagger+n,
\end{equation}
then by subtracting $\pd_QR^\prime$ and taking the infimum over all quasi-deformations, we would obtain \eqref{eq:ci}. We denote by $\bar{\x}$ the sequence of images of $x_1,\ldots,x_n$ in $R^\prime$ and by $\y=y_1,\ldots,y_n$ lifts of these images to $Q$. Let $K^{R^\prime}(\bar{\x})$ and $K^Q(\y)$ be the Koszul complexes on $\bar{\x}$ and $\y$ respectively. Then,
\[
R^\prime\lotimes_QK^Q(\y)\simeq K^{R^\prime}(\bar{\x})\simeq R^\prime\lotimes_R K,
\]
therefore
\begin{align*}
R^\prime\lotimes_R(K\lotimes_R M^\dagger)&\simeq(R^\prime\lotimes_R K)\lotimes_{R^\prime}(R^\prime\lotimes_R M^\dagger)\\
&\simeq(R^\prime\lotimes_Q K^Q(\y))\lotimes_{R^\prime}(R^\prime\lotimes_R M^\dagger)\\
&\simeq K^Q(\y)\lotimes_Q(R^\prime\lotimes_R M^\dagger)
\end{align*}
By \cite[Proposition 12.1.20(c)]{LarsBook} $R^\prime\lotimes_RM^\dagger\in\Dfb$, therefore \cite[Proposition 16.4.17]{LarsBook} justifies the second equality below
\begin{align*}
\pd_Q R^\prime\lotimes_R(K\lotimes_R M^\dagger)&=\pd_Q K^Q(\y)\lotimes_Q(R^\prime\lotimes_RM^\dagger)\\
&=\pd_QK^Q(\y)+\pd_Q R^\prime\lotimes_RM^\dagger\\
&=\pd_Q R^\prime\lotimes_RM^\dagger+n,
\end{align*}
this proves \eqref{eq:ci2} and therefore \eqref{eq:ci}. The string of (in)equalities below yields $\cidim_R M^\dagger\leq\CIHid M$
\begin{align*}
\cidim_R M^\dagger+n&=\cidim_R K\lotimes_R M^\dagger\\
&=\cidim_R Z^\vee\\
&=\CIHid_R Z\\
&\leq\CIHid_R M+n,
\end{align*}
where the first equality follows from \eqref{eq:ci}, the second from \eqref{eq:iso}, the third from \cite[Theorem 4.5(b)]{cihid} which can be applied since $Z$ has finite length homology, and the inequality follows from \cite[Proposition 4.3]{cihid}.

For the reverse inequality we note
\begin{align*}
\CIHid_R M&=\CIHid_R\RHom_R(M^\dagger, D)\\
&\leq\cidim_R M^\dagger+\id_RD\\
&=\cidim_R M^\dagger, 
\end{align*}
where the first equality follows from Grothendieck duality, the inequality from \cite[Proposition 4.1]{cihid}, and the last equality from \cite[Theorem 18.2.24]{LarsBook}.
\end{proof}

\begin{cor}\label{cor:ciinj}
Let $R$ be a local ring with a dualizing complex, and let $M\in\Dfb$. Then,
\[
\ciuid_RM=\CIHid_RM=\ciid_RM.
\]
\end{cor}
\begin{proof}
Let $D$ be a normalized dualizing complex. By \cite[Corollary 4.6(b)]{ciInj} $\ciuid_RM<\infty$ if and only if $\cidim_RM^\dagger<\infty$. If both quantities are finite then
\[
\ciuid_RM=\depth R-\inf M
\]
by Bass' Formula \cite[Proposition 2.11]{ciInj}, and
\[
\cidim_RM^\dagger=\depth R-\depth M^\dagger=\depth R-\inf M,
\]
where the first equality is the Auslander-Buchsbaum formula for complete intersection dimension, and the second equality follows from \cite[Theorem 18.2.31(b)]{LarsBook}. Therefore
\[
\ciuid_RM=\cidim_RM^\dagger.
\]
Invoking \Cref{thm:duality}(1) yields
\[
\ciuid_RM=\CIHid_RM.
\]

Next we show that $\ciid_RM=\ciuid_RM$. The inequality $\ciid_RM\leq\ciuid_RM$ is obvious, and when the right side is finite one gets equality since both dimensions satisfy a Bass formula, see \cite[Proposition 2.11]{ciInj}.

For the reverse inequality we notice that if $\ciid_RM<\infty$, then $\ciid_RM^{\dagger\dagger}<\infty$. By \cite[Proposition 3.3(3)]{MMRN} $\cidim_RM^\dagger<\infty$, and by \cite[Corollary 4.6(b)]{ciInj} $\ciuid_RM<\infty$, and since both dimensions satisfy the Bass formula one has equality.
\end{proof}

Next we explore the relation between the three injective versions of complete intersection dimension for rings that do not necessarily admit a dualizing complex. We first study the behavior of these dimensions under completion. For the upper complete intersection injective dimension the answer already follows from \cite[Proposition 2.11(b) and Corollary 3.7(a)]{ciInj}: if $R$ is a local ring and $M\in\Dfb$, then
\[
\ciuid_RM=\ciuid_{\wh{R}}\wh{M}.
\]
Part (1) of the next result shows that \cite[Corollary 3.7(b)]{ciInj} holds for complexes in $\Dfb$ that do not necessarily have finite length homology.

\begin{theorem}\label{thm:completion}
If $R$ is a local ring and $M\in\Dfb$, then
\begin{enumerate}
\item $\ciid_{\wh R}\wh M=\ciid_RM$.
\item $\CIHid_{\wh R}\wh M\leq\CIHid_RM$, with equality if the right hand side is finite.
\end{enumerate}
\end{theorem}
\begin{proof}
\begin{enumerate}
\item We first assume $\ciid_RM<\infty$. Let $R\rightarrow R'\leftarrow Q$ be a quasi-deformation such that $\id_Q R'\lotimes_RM<\infty$. We note that $R'\lotimes_RM\in\Dfb[R']$ and therefore it belongs to $\Dfb[Q]$ since $R'$ is finitely generated as a $Q$-module. By \cite[Proposition 18.3.10]{LarsBook}
\[
\id_{\wh Q}\wh{Q}\lotimes_QR'\lotimes_RM=\id_QR'\lotimes_RM<\infty.
\]
The completion of the quasi-deformation considered above gives a quasi-deformation of $\wh R$
\[
\wh R\rightarrow\wh{R'}\leftarrow\wh{Q},
\]
indeed, the left map is flat by \cite[Theorem 22.4(i)]{matsu}, and the right map is a deformation by \cite[Corollary 1.1.3(b)]{BrunsHerzog}. We note that
\begin{align*}
\wh{R'}\lotimes_{\wh R}\wh M&\simeq\wh{R'}\lotimes_{\wh R}\wh R\lotimes_RM\\
&\simeq\wh{R'}\lotimes_RM\\
&\simeq\wh{Q}\lotimes_Q\left(R'\lotimes_RM\right).
\end{align*}
This shows that $\id_{\wh Q}\wh{R'}\lotimes_{\wh R}\wh M<\infty$, and therefore $\ciid_{\wh R}\wh M<\infty$.

We now assume $\ciid_{\wh R}\wh M<\infty$. Since $\wh R$ has a dualizing complex, it follows from \Cref{cor:ciinj} that $\ciuid_{\wh R}\wh M<\infty$, and therefore by \cite[Corollary 3.7(a)]{ciInj} $\ciuid_RM<\infty$, which trivially implies that $\ciid_RM<\infty$.

If both dimensions are finite, then
\begin{align*}
\ciid_RM&=\depth R-\inf M\\
&=\depth \wh R-\inf\wh M\\
&=\ciid_{\wh R}\wh M,
\end{align*}
where the first and third equality follow from \cite[Proposition 2.11(a)]{ciInj}.
\item We assume $\CIHid_RM<\infty$. Let $E_R(\kk)$ be the injective hull of $\kk$ over $R$, and set 
\[
M^\vee\colonequals\RHom_R(M,E_R(\kk)).
\]
By \cite[Theorem 4.5(a)]{cihid} $\cifd_RM^\vee<\infty$. By \cite[Proposition 4.2]{cifd}
\[
\cifd_{\wh R}\wh{R}\lotimes_RM^\vee<\infty.
\]
Since $M^\vee$ is derived $\m$-torsion, it follows from \cite[Theorem 13.4.16(c)]{LarsBook} that $\wh{R}\lotimes_RM^\vee\simeq M^\vee$. By \cite[Exercise 3.2.14]{BrunsHerzog} $E$ is also the injective hull of $\kk$ over $\wh R$, denoted by $E_{\wh R}(\kk)$. Therefore, adjunction yields
\[
\RHom_{\wh R}(\wh M,E_{\wh R}(\kk))\simeq M^\vee.
\]
By \cite[Proposition 18.3.2]{LarsBook} $\wh M\in\Dfb[\wh R]$, and therefore by \cite[Theorem 18.1.9]{LarsBook} $\wh M$ is a derived Matlis reflexive object of $\D(\wh R)$. Therefore,
\begin{align*}
\CIHid_{\wh R}\wh M&=\cifd_{\wh R}\RHom_{\wh R}(\wh M, E_{\wh R}(\kk))\\
&=\cifd_{\wh R} M^\vee\\
&=\cifd_{\wh R}\wh R\lotimes_RM^\vee<\infty,
\end{align*}
where the first equality follows from \cite[Theorem 4.5(b)]{cihid}, which can be applied by \cite[Lemma 16.1.40]{LarsBook}.

In particular
\begin{align*}
\CIHid_{\wh R}\wh M&=\depth \wh R-\inf\wh M\\
&=\depth R-\inf M\\
&=\CIHid_RM,
\end{align*}
where the first and third equality follow from \cite[Corollary 3.4]{cihid}.\qedhere
\end{enumerate}
\end{proof}
\begin{question}
Let $R$ be a local ring and $M\in\Dfb$. If $\CIHid_{\wh R}\wh M$ is finite, is $\CIHid_RM$ finite as well?
\end{question}
Since $\wh R$ has a dualizing complex, an immediate consequence of \Cref{cor:ciinj}, \cite[Proposition 2.11 and Corollary 3.7
(a)]{ciInj}, and \Cref{thm:completion} is the following
\begin{cor}\label{cor:ciinj=}
Let $R$ be a local ring and $M\in\Dfb$. Then
\[
\ciuid_RM=\ciid_RM.
\]
Moreover, if $\CIHid_RM<\infty$, then
\[
\ciuid_RM=\ciid_RM=\CIHid_RM.
\]
\end{cor}

This answers \cite[Question 2.9]{ciInj} for complexes in $\Dfb$. We give an example showing that if $M$ is bounded but not with finitely generated homology, then it is possible for $\ciuid_RM=\infty$ and $\ciid_RM<\infty$, providing a full answer to Sather-Wagstaff's question. We start with the following 
\begin{lemma}\label{lem:ex}
Let $R$ be a Noetherian local domain with fraction field $K$. If $\wh{R}\otimes_RK$ is not Gorenstein, then
\[
\ciid_RK=0,\quad\text{and}\quad\ciuid_RK=\infty.
\]
\end{lemma}
\begin{proof}
By Baer's criterion the fraction field of a domain is always injective, therefore $K$ is an injective $R$-module yielding $\ciid_RK=0$.

Suppose that $\ciuid_RK<\infty$, then \cite[Proposition 3.5]{ciInj} shows that there exists a quasi-deformation $R\rightarrow R'\xleftarrow{\tau}Q$ such that $Q$ is complete and $\id_Q R'\otimes_RK<\infty$. By faithful flatness of $R\rightarrow R'$ and \cite[Theorem 7.3(i)]{matsu}, there is $\p'\in\Spec R'$ such that $\p'\cap R=(0)$. Let $\q\colonequals\tau^{-1}(\p')$. We note that $(R'\otimes_RK)_{\p'}\cong R'_{\p'}$. Therefore, since $\id_QR'\otimes_RK<\infty$, it follows that $\id_{Q_\q}R'_{\p'}<\infty$. Since $R'_{\p'}$ is a deformation of $Q_\q$, it follows that $\pd_{Q_\q}R'_{\p'}<\infty$, and since $R'_{\p'}$ is a finitely generated $Q_\q$-module, we deduce that $Q_\q$ is Gorenstein. Since $R'_{\p'}$ is a quotient of $Q_\q$ by a regular sequence, it follows that it is also Gorenstein. A prime ideal of $R'\otimes_RK$ corresponds to a prime ideal of $R'$ that when contracted to $R$ gives $(0)$. The localization of $R'\otimes_RK$ at this prime ideal is isomorphic to the localization of $R'$ at the contraction. Since we proved that these localizations are Gorenstein, it follows that $R'\otimes_RK$ is Gorenstein. Consider the map $\wh{R}\rightarrow\wh{R'}=R'$, which is faithfully flat by \cite[Theorem 22.4(i)]{matsu}. Tensoring this map with $K$ yields a faithfully flat map $\wh{R}\otimes_RK\rightarrow R'\otimes_RK$ by \cite[Corollary 5.4.24]{LarsBook}. By \cite[Theorem 17.4.15]{LarsBook} it follows that $\wh{R}\otimes_RK$ is Gorenstein, contradicting the hypothesis.
\end{proof}
\begin{rem}
In \cite[Proposition 3.1 and Remark 3.2(i)]{FerRay}, Ferrand and Raynaud show that a Noetherian local domain $R$ with fraction field $K$ such that $\wh{R}\otimes_RK$ is not Gorenstein does indeed exist. This, together with \Cref{cor:ciinj=} and \Cref{lem:ex}, provides a full answer to Sather-Wagstaff's question \cite[Question 2.9]{ciInj}.
\end{rem}

As another application of \Cref{lem:ex} we prove

\begin{cor}
Let $R$ be a Noetherian local domain of Krull dimension 1 and with fraction field $K$. Then, $\ciuid_RK<\infty$ if and only if $R$ admits a canonical ideal.
\end{cor}
\begin{proof} First assume $\ciuid_RK<\infty$.
Let $S=R\backslash\{0\}$ and let $T$ be the set of nonzero divisors of $\wh R$. By \Cref{lem:ex} and \cite[Proposition 2.7]{almostG} it suffices to show that $S^{-1}\wh{R}\cong T^{-1}\wh R$. Since $\wh R$ is flat over $R$, it is torsion-free, and it follows that $S\subseteq T$. This provides an injection $\varphi:S^{-1}\wh R\rightarrow T^{-1}\wh R$. We note that by \cite[Page 110]{olb} $S^{-1}\wh R$ is Artinian. Let $x\in T$, then $x/1$ is a nonzero divisor in $S^{-1}\wh R$ since $\widehat R$ is torsion-free over $R$. Since $S^{-1}\wh R$ is Artinian, it follows that $x/1$ is a unit. By the universal property of localization, this gives a map $\psi:T^{-1}\wh R\rightarrow S^{-1}\wh R$. One can verify that $\psi$ is the inverse of $\varphi$.

Conversely, if $R$ admits a canonical module $D$, then, since $\id_R K<\infty$, we have $\fd_R \RHom_R(K, D)<\infty$, hence $\cifd_R \RHom_R(K,D)<\infty$. Thus, $\ciuid_R K<\infty$ by \cite[Corollary 4.6]{ciInj}.
\end{proof}

In \cite[Question 3.9]{cihid}, the authors ask whether for a local ring $R$, $M\in\Db$ and $\p\in\Spec R$, one has
\[
\CIHid_{R_\p}M_\p\leq\CIHid_RM.
\]
We answer this question when $M\in\Dfb$ and $R$ has a dualizing complex. If $\CIHid_RM=\infty$, then there is nothing to prove. Otherwise, by \Cref{cor:ciinj}, it suffices to show the proposition below, which is an injective version of \cite[Proposition 1.6]{cidim} and is proved similarly.
\begin{proposition}
Let $R$ be a commutative Noetherian ring not necessarily local and $M\in\Dfb$. Then,
\[
\ciid_RM=\sup\{\ciid_{R_\p}M_\p\mid \p\in \Supp_RM\}.
\]
\end{proposition}
\begin{proof}
It suffices to show that when $R$ is local and $\p\in\Supp_RM$ one has $\ciid_{R_\p}M_\p\leq\ciid_RM$. Assume that $\ciid_RM<\infty$. Let $R\rightarrow R'\xleftarrow{\tau} Q$ be a quasi-deformation with $\id_Q R'\lotimes_R M<\infty$. By \cite[Theorem 7.3(i)]{matsu} there is $\p'\in\Spec R'$ such that $\p'\cap R=\p$. Let $\q\colonequals\tau^{-1}(\p')$. Then $R_\p\rightarrow R'_{\p'}\leftarrow Q_\q$ is a quasi-deformation of $R_\p$. Therefore,
\begin{align*}
\ciid_{R_\p}M_\p&\leq\id_{Q_\q}(R'_{\p'}\lotimes_{R_\p}M_\p)-\pd_{Q_\q}R'_{\p'}\\
&=\id_{Q_\q}(R'\lotimes_RM)_\q-\pd_{Q_\q}R'_{\p'}\\
&\leq\id_QR'\lotimes_RM-\pd_{Q_\q}R'_{\p'}\\
&=\id_QR'\lotimes_RM-\pd_QR'.
\end{align*}
Taking the infimum over all quasi-deformations of $R$ yields the desired inequality.
\end{proof}

\bibliographystyle{amsplain}
\bibliography{biblio}
\end{document}